\documentclass{article}
\usepackage{amsmath}
\usepackage{amsfonts}
\usepackage{amsthm}
\usepackage{centernot}
\usepackage{url}

\usepackage{enumitem}

\def\intZ{ \mathbb{Z} }

\def\natN{  \mathbb{N} }

\def\notmid{\centernot\mid}

\newtheorem{theorem}{Theorem}[section]
\newtheorem{lemma}{Lemma}[section]
\newtheorem{conjecture}{Conjecture}[section]
\newtheorem{corollary}{Corollary}[section]

\title{An application of the BHV theorem to a new conjecture on exponential diophantine equations  }
\author{Maohua Le }
\date{17 October 2018}

\begin{document}

\maketitle

\begin{abstract}
Let $A$, $B$ be fixed positive integers such that $\min\{A,B\} > 1$, $\gcd(A,B) = 1$ and $AB \equiv 0 \bmod 2$, and let $n$ be a positive integer with $n>1$.  In this paper, using a deep result on the existence of primitive divisors of Lucas numbers due to Y. Bilu, G. Hanrot and P. M. Voutier \cite{BHV}, we prove that if $A > 8 B^3$, then the equation $(*) \quad (A^2 n)^x + (B^2 n)^y = ((A^2 + B^2)n)^z$ has no positive integer solutions $(x,y,z)$ with $x > z > y$.  Combining the above conclusion with some existing results, we can deduce that if $A >8 B^3$ and $B \equiv 2 \bmod 4$, then (*) has only the positive integer solution $(x,y,z) = (1,1,1)$.    
\end{abstract}

\bigskip



\section{Introduction}  

Let $\intZ$, $\natN$ be the sets of all integers and positive integers, respectively.  Let $a$, $b$ be fixed positive integers such that $\min\{a,b\} > 1$ and $\gcd(a,b) =1$, and let $n$ be a positive integer with $n>1$.  Recently, P.-Z. Yuan and Q. Han \cite{YuHa} proposed the following conjecture:

\begin{conjecture}  
If $\min\{a,b\} \ge 4$, then the equation
$$ (an)^x + (bn)^y = ((a+b)n)^z, x,y,z \in \natN  \eqno{(1.1)} $$
has only the solution $(x,y,z) = (1,1,1)$.  
\end{conjecture} 

Since Conjecture 1.1 is much broader than Je\'smanowicz' conjecture concerning Pythagorean triples (see \cite{Jes}), it is unlikely to be solved in the short term.  There are only a few scattered results on Conjecture 1.1 at present (see \cite{SuTa}).  In the same paper, P.-Z. Yuan and Q. Han \cite{YuHa} deal with the solutions $(x,y,z)$ of (1.1) for the case that $(a,b) = (A^2, B^2)$, where $A$, $B$ are fixed positive integers such that $\min\{A,B\} > 1$, $\gcd(A,B)=1$ and $AB \equiv 0 \bmod 2$.  Then (1.1) can be rewritten as 
$$ (A^2 n)^x + (B^2 n)^y = ((A^2 + B^2)n)^z, x,y,z \in \natN.  \eqno{(1.2)} $$
In this respect, they proved that if $B \equiv 2 \bmod 4$, then (1.2) has no solutions $(x,y,z)$ with $y>z>x$, in particular, if $B=2$, then Conjecture 1.1 is true for $(a,b) = (A^2, B^2)$.  

In 2001, Y. Bilu, G. Hanrot and P. M. Voutier \cite{BHV} completely solved the existence of primitive divisors of Lucas numbers.  This result is usually called the BHV theorem.  In this paper, using the BHV theorem, we prove the following result:

\begin{theorem}  
If $A > 8 B^3$, then (1.2) has no solutions $(x,y,z)$ with $x>z>y$.  
\end{theorem}

By Theorem 1.1 and some results of \cite{SuTa} and \cite{YuHa}, we can deduce the following corollary:

\begin{corollary}  
If $A > 8 B^3$ and $B \equiv 2 \bmod 4$, then (1.2) has only the solution $(x,y,z)=(1,1,1)$. 
\end{corollary}

This implies that, for any fixed $B$ with $B \equiv 2 \bmod 4$, then Conjecture 1.1 is true for $(a,b) = (A^2, B^2)$ except for finitely many values of $A$.

\section{Preliminaries}  

Let $\alpha$, $\beta$ be algebraic integers.  If $\alpha + \beta$ and $\alpha \beta$ are nonzero coprime integers and $\alpha/\beta$ is not a root of unity, then $(\alpha, \beta)$ is called a Lucas pair.  Further, let $u = \alpha + \beta$ and $w = \alpha \beta$.  Then we have 
$$ \alpha = \frac{1}{2} (u + \lambda \sqrt{v}), \beta = \frac{1}{2} (u - \lambda \sqrt{v}), \lambda \in \{-1,1\}, $$
where $v = u^2 - 4w$.  We call $(u,v)$ the parameters of the Lucas pair $(\alpha, \beta)$.  Two Lucas pairs $(\alpha_1, \beta_1)$ and $(\alpha_2, \beta_2)$ are equivalent if $\alpha_1/\alpha_2 = \beta_1/\beta_2 = \pm 1$.  Given a Lucas pair $(\alpha, \beta)$, one defines the corresponding sequence of Lucas numbers by 
$$ L_n(\alpha,\beta) = \frac{\alpha^n - \beta^n}{\alpha - \beta}, n = 0, 1, \dots.  \eqno{(2.1)} $$
Obviously, if $n>0$, then the Lucas numbers $L_n(\alpha, \beta)$ are nonzero integers.  For equivalent Lucas pairs $(\alpha_1, \beta_1)$ and $(\alpha_2, \beta_2)$, we have $L_n(\alpha_1,\beta_1) = \pm L_n(\alpha_2,\beta_2)$ for any $n \ge 0$.  A prime $p$ is called a primitive divisor of $L_n(\alpha,\beta)$ ($n > 1$) if 
$$ p \mid L_n(\alpha,\beta), p \notmid v L_1(\alpha,\beta)\dots L_{n-1}(\alpha,\beta).  \eqno{(2.2)} $$
A Lucas pair $(\alpha, \beta)$ such that $L_n(\alpha,\beta)$ has no primitive divisors will be called an $n$-defective Lucas pair.  Further, a positive integer $n$ is called totally non-defective if no Lucas pair is $n$-defective.  

\begin{lemma}[P. M. Voutier \cite{Vou}]  
Let $n$ satisfy $4 < n \le 30$ and $n \ne 6$.  Then, up to equivalence, all parameters of $n$-defective Lucas pairs are given as follows:
\begin{enumerate}[label=(\roman*)]
  \item $n=5$, $(u,v) = (1,5)$, $(1,-7)$, $(2,-40)$, $(1,-11)$, $(1, -15)$, $(12,-76)$, $(12, -1364)$. 
  \item $n=7$, $(u,v) = (1, -7)$, $(1, -19)$.
  \item $n=8$, $(u,v) = (2, -24)$, $(1, -7)$.
  \item $n=10$, $(u,v) = (2, -8)$, $(5,-3)$, $(5, -47)$.
  \item $n=12$, $(u,v) = (1, 5)$, $(1, -7)$, $(1,-11)$, $(2, -56)$, $(1, -15)$, $(1, -19)$.
  \item $n \in \{ 13, 18, 30\}$, $(u,v) = (1, -7)$.
 \end{enumerate}   
\end{lemma} 

\begin{lemma}[Y. Bilu, G. Hanrot, P. M. Voutier \cite{BHV}]  
If $n>30$ then $n$ is totally non-defective.  
\end{lemma}

Let $D$, $k$ be fixed positive integers such that $\min\{ D, k\} > 1$ and $\gcd(2D, k) = 1$, and let $h(-4D)$ denote the class number of positive definite binary quadratic primitive forms of discriminant $-4D$.  

\begin{lemma} 
If the equation 
$$ X^2 + D Y^2 = k^Z, X,Y,Z \in \intZ, \gcd(X,Y)=1, Z>0 \eqno{(2.3)} $$
has solutions $(X,Y,Z)$, then every solution $(X,Y,Z)$ of (2.3) can be expressed as 
$$Z = Z_1 t, t \in \natN, \eqno{(2.4)} $$
$$X+Y \sqrt{-D} = \lambda_1 (X_1 + \lambda_2 Y_1 \sqrt{-D})^t, \lambda_1, \lambda_2 \in \{ -1, 1\}, \eqno{(2.5)} $$
where $X_1$, $Y_1$, $Z_1$ are positive integers satisfying $X_1^2 + D Y_1^2 = k^{Z_1}$, $\gcd(X_1, Y_1) = 1$ and 
$$ h(-4D) \equiv 0 \bmod Z_1.  \eqno{(2.6)} $$
\end{lemma} 

\begin{proof}  
This lemma is a special case of Theorems 1 and 2 of \cite{Le} for $D_1 = 1$ and $D_2 < -1$.  
\end{proof}

\begin{lemma}[Theorems 12.10.1 and 12.14.3 of \cite{Hua}]  
$h(-4D) < \frac{4}{\pi} \sqrt{D} \log(2 e \sqrt{D})$.  
\end{lemma}

For a positive integer $m$ with $m>1$, let $m = p_1^{t_1} \dots p_l^{t_l}$ denote the factorization of $m$.  Further, let $r(m) = p_1 \dots p_l$, 
$$ R(m) = p_1^{e_1} \dots p_l^{e_l}, e_i = 
\begin{cases}
2, {\rm if \ } t_i \equiv 0 \bmod 2, \\  
1, {\rm if \ } t_1 \equiv 1 \bmod 2, 
\end{cases}
i = 1, \dots, l, $$
$r(1) = R(1) = 1$ and 
$$ S(m) = \{ \pm p_1^{s_1} \dots p_l^{s_l} \mid s_i \in \intZ, s_i \ge 0, i = 1, \dots, l \}. $$
Obviously, $r(m)$ and $R(m)$ have the following properties:
\begin{enumerate}[label=(\roman*)]
  \item $R(m)>1$ if $m>1$.
  \item $m / R(m)$ is the square of a positive integer. 
  \item For any positive integer $m'$ with $m' \in S(m)$, $R(m') \le (r(m))^2$.
  \item For any coprime positive integers $m_1$ and $m_2$, $r(m_1 m_2) = r(m_1) r(m_2)$ and $R(m_1 m_2) = R(m_1) R(m_2)$.  
\end{enumerate}  

\begin{lemma}  
If $D>2$ and $(X,Y,Z)$ is a solution of (2.3) with 
$$ Y \in S(D), \eqno{(2.7)} $$
then 
$$ Z \le 6 h(-4D). \eqno{(2.8)} $$
\end{lemma} 

\begin{proof}
By Lemma 2.4, the solution $(X,Y,Z)$ can be expressed as (2.4) and (2.5), where $X_1$, $Y_1$, $Z_1$ are positive integers satisfying $\gcd(X_1, Y_1)=1$ and (2.6).  Let 
$$ \alpha = X_1 + Y_1 \sqrt{-D}, \beta = X_1 - Y_1 \sqrt{-D}. \eqno{(2.9)} $$
Then we have 
$$ \alpha + \beta = 2 X_1, \alpha \beta = k^{Z_1}, \eqno{(2.10)} $$
$$ \frac{\alpha}{\beta} =  \frac{1}{k^{Z_1}} ( (X_1^2 - D Y_1 ^2) + 2 X_1 Y_1 \sqrt{-D} ).  \eqno{(2.11)}  $$
Since $\gcd(X_1, Y_1) = \gcd(2D, k) = 1$, we see from (2.10) that $\alpha + \beta$ and $\alpha \beta$ are coprime postive integers.  Since $k>1$, by (2.11), $\alpha/\beta$ is not a root of unity.  Hence, by (2.9), $(\alpha, \beta)$ is a Lucas pair with parameters 
$$(u,v) = (2 X_1, -4 D Y_1^2 ).  \eqno{(2.12)} $$

Let $L_n(\alpha, \beta)$ ($n = 0, 1, \dots$) denote the corresponding Lucas numbers.  Since $X - Y \sqrt{-D} = \lambda_1 (X_1 - \lambda_2 Y_1 \sqrt{-D})^t$ by (2.5), we get from (2.1), (2.5) and (2.9) that 
$$Y = Y_1 \vert L_t(\alpha, \beta) \vert.  \eqno{(2.13)} $$
Hence, we find from (2.2), (2.7), (2.12) and (2.13) that either $t=1$ or $t>1$ and $L_t(\alpha, \beta)$ has no primitive divisors.  Therefore, by Lemma 2.2, we have $t \le 30$.  Further, since $D>2$, applying Lemma 2.1 to (2.12), we get either
$$ t \le 6  \eqno{(2.14)} $$
or
$$ ( D, k, X_1, Y_1, Z_1, t) = (6,7,1,1,1,8), (14,15,1,1,1,12). \eqno{(2.15)} $$
Since $Z_1 \le h(-4D)$ by (2.6), if $t$ satisfies (2.14), then $Z$ satisfies (2.8) by (2.4).  On the other hand, since $h(-24) = 2$ and $h(-56) = 4$, if $t$ satisfies (2.15), then $Z$ also satisfies (2.8).  Thus, the lemma is proven.  
\end{proof}

\begin{lemma}[C.-F. Sun and M. Tang \cite{SuTa}] 
Let $(x,y,z)$ be a solution of (1.1) with $(x,y,z) \ne (1,1,1)$.  If $\min\{a,b\} \ge 4$, then either 
$$ x>z>y, b = b_1 b_2, b_1^{y} = n^{z-y}, b_1, b_2 \in \natN, b_1 >1, \gcd(b_1, b_2) = 1 $$
or
$$ y>z>x, a = a_1 a_2, a_1^{x} = n^{z-x}, a_1, a_2 \in \natN, a_1>1, \gcd(a_1, a_2)=1. $$
\end{lemma}

\section{Proof of Theorem and Corollary}  

\begin{proof}[Proof of Theorem 1.1]
Let $(x,y,z)$ be a solution of (1.2) with $x >z>y$.  By Lemma 2.6, we have 
$$ B = B_1 B_2, B_1, B_2 \in \natN, B_1 >1, \gcd(B_1, B_2)=1 \eqno{(3.1)} $$
and 
$$ B^{2y} = n^{z-y}.  \eqno{(3.2)} $$
Substituting (3.1) and (3.2) into (1.2), we get
$$ A^{2x} n^{x-z} + B_2^{2y} = (A^2 + B^2)^z.  \eqno{(3.3)} $$

Let $D = R(A^{2x} n^{x-z} )$.  Since $A^{2x} n^{x-z} / R(A^{2x} n^{x-z}) $ is the square of a positive integer, we see from (3.3) that the equation 
$$X^2 + D Y^2 = (A^2 + B^2)^Z, X,Y,Z \in \intZ, \gcd(X,Y) =1, Z>0  \eqno{(3.4)} $$
has the solution 
$$ (X,Y,Z) = \left( B_2^y, \sqrt{\frac{A^{2x} n^{x-z}}{D}}, z \right)  \eqno{(3.5)} $$
with (2.7) holding.  

Since $\gcd(A,B) = 1$, by (3.1) and (3.2), we have $\gcd(A, B_1) = \gcd(A,n) = 1$.  Hence, we get 
$$ D = R(A^{2x} n^{x-z}) = R(A^{2x}) R(n^{x-z}).  \eqno{(3.6)} $$
Further, since $R(A^{2x})>1$ and $R(n^{x-z})>1$, by (3.6), we have have $D>2$.  Since $\gcd(A,B)=1$ and $AB \equiv 0 \bmod 2$, we have $\gcd(2D, A^2 + B^2) = 1$.  Therefore, applying Lemma 2.5 to (3.5), we get
$$ z \le 6 h(-4D).  \eqno{(3.7)} $$

Since $r(n)= r(B_1)$ and $n^{x-z} \in S(B_1)$ by (3.2), we have $R(A^{2x}) \le A^2$ and $R(n^{x-z}) \le B_1^2$.  Hence, by (3.6), we get
$$ D \le A^2 B_1^2. \eqno{(3.8)} $$
Further, by Lemma 2.4, we obtain from (3.8) that 
$$ h(-4D) < \frac{4}{\pi} A B_1 \log(2 e A B_1).  \eqno{(3.9)} $$

On the other hand, since $x>z$, by (3.3), we have 
$$ (A^2 n)^{x-z} < \left(1 + \frac{B^2}{A^2} \right)^z.  \eqno{(3.10)} $$
Since $\log(1+\theta) < \theta$ for any $\theta >0$, by (3.10), we get
$$ (x-z) \log(A^2 n) < \frac{x B^2 }{A^2}.  \eqno{(3.11)} $$
Further, since $A > 8 B^3$, by (3.11), we obtain
$$ z > (x-z) \frac{A^2}{B^2} \log(A^2 n) \ge \frac{A^2}{B^2} \log(A^2 n) > 8 A B \log(A^2 n).  \eqno{(3.12)} $$
The combination of (3.7), (3.9) and (3.12) yields 
$$ \frac{24}{\pi} A B_1 \log( 2 e A B_1) > 8 A B \log(A^2 n).  \eqno{(3.13)} $$
But, since $24/\pi < 8$, $A B_1 \le A B$ and $2 e A B_1 < 8 A B^3 < A^2 < A^2 n$, (3.13) is false.  Thus, if $A > 8 B^3$, then (1.2) has no solutions $(x,y,z)$ with $x>z>y$.  
\end{proof}

\begin{proof}[Proof of Corollary 1.1]
Under the assumptions, by Theorem 1.1 and the result of \cite{YuHa}, (1.2) has no solutions with $x>z>y$ nor with $y>z>x$.  Therefore, by Lemma 2.6, (1.2) has no solutions $(x,y,z)$ with $(x,y,z) \ne (1,1,1)$.  
\end{proof}

\bigskip

\noindent
Acknowledgment: The author would like to thank Prof. R. Scott and Prof. R. Styer for reading the original manuscript carefully and giving valuable advice.  Especially thanks to Prof. R. Styer for verifying some details of this paper and providing other technical assistance.

\end{document}